\documentclass[11pt,a4paper]{amsart}

\usepackage[utf8]{inputenc}
\usepackage[american]{babel}
\usepackage{amsmath,amssymb,amsthm}
\usepackage{csquotes}
\usepackage{calrsfs}
\usepackage{dsfont}
\usepackage[style=alphabetic,sorting=anyt,backend=biber,doi=false,url=false,maxbibnames=10,isbn=false]{biblatex}
\usepackage[left=3cm,right=3cm,top=3cm,bottom=3cm]{geometry}
\usepackage[hidelinks]{hyperref}

\newcommand{\1}{\mathds{1}}

\newcommand{\n}{\mathfrak n}

\newcommand{\q}{\mathfrak q}

\renewcommand{\epsilon}{\varepsilon}
\renewcommand{\phi}{\varphi}

\renewcommand{\AA}{\mathfrak A}

\newcommand{\HH}{\mathfrak H}

\newcommand{\E}{\mathcal E}
\renewcommand{\H}{\mathcal H}

\newcommand{\J}{\mathcal J}

\newcommand{\U}{\mathcal U}

\newcommand{\IC}{\mathbb C}
\newcommand{\IN}{\mathbb N}
\newcommand{\IR}{\mathbb R}
\newcommand{\IZ}{\mathbb Z}

\newcommand{\dom}{\operatorname{dom}}
\newcommand{\id}{\mathrm{id}}
\newcommand{\op}{\mathrm{op}}

\newcommand{\abs}[1]{\lvert#1\rvert}
\newcommand{\norm}[1]{\lVert#1\rVert}

\theoremstyle{plain}
\newtheorem{proposition}{Proposition}[section]
\newtheorem{lemma}[proposition]{Lemma}
\newtheorem{theorem}[proposition]{Theorem}
\newtheorem*{theorem*}{Theorem}

\theoremstyle{definition}

\theoremstyle{remark}
\newtheorem{remark}[proposition]{Remark}
\newtheorem{example}[proposition]{Example}

\title[Dirichlet Forms and Derivations]{Modular Completely Dirichlet forms as Squares of Derivations}
\author{Melchior Wirth}
\address{Institute of Science and Technology Austria (ISTA), Am Campus 1, 3400 Klosterneuburg, Austria}
\email{melchior.wirth@ist.ac.at}

\begin{document}

\begin{abstract}
We prove that certain closable derivations on the GNS Hilbert space associated with a non-tracial weight on a von Neumann algebra give rise to GNS-symmetric semigroups of contractive completely positive maps on the von Neumann algebra.
\end{abstract}

\maketitle

\section*{Introduction}

The interplay between derivations and symmetric semigroups of unital (or contractive) completely positive maps has proven fruitful for applications in quantum information theory \cite{Bar17,CM17,GR22}, operator algebras \cite{Pet09a,Pet09b,DI16,Cas21} and beyond. Using the framework of completely Dirichlet forms, this connection is particularly well-understood in the case of tracially symmetric semigroups after the seminal work of Cipriani and Sauvageot \cite{Sau89,Sau90,CS03}.

In many situations however one encounters non-tracial reference states or weights: In quantum statistical mechanics, the reference state is typically a Gibbs state, which is not a trace at finite temperature; in quantum probability in the study of Lévy processes on compact quantum groups, the natural reference state is the Haar state, which is only a trace for the class of compact quantum groups of Kac type; and in the structure theory of von Neumann algebras, one is faced with non-tracial states when the von Neumann algebra has a non-trivial type III summand.

In the non-tracial setting, the connection between derivations and symmetric semigroups of completely positive maps is much less understood. Recently, it was shown by the author that every GNS-symmetric semigroup of unital completely positive maps gives rise to a canonical derivation via its associated Dirichlet form \cite{Wir22b}.  This result was (partially) extended to KMS-symmetric semigroups by Vernooij and the author \cite{VW23}.

There has also been work in the opposite direction -- starting with a derivation to construct a completely Dirichlet form \cite{Cip97,Par00,BK02,CZ21,LM22,Wir22b}. However, these results all rely on additional structural assumptions on the derivation, usually some form of (approximate) innerness. This means that natural examples like derivations arising from cocycles on non-unimodular groups or Voiculescu's derivation in non-tracial free probability could not be treated in this framework.

In this article, we prove in a general context that closable derivations give rise to GNS-symmetric semigroups of completely bounded maps. More precisely, our main result is the following.

\begin{theorem*}
Let $\AA$ be a Tomita algebra, $\H$ a normal Tomita bimodule over $\AA$ and $\delta\colon \AA\to\H$ a closable symmetric derivation. Let $\mathcal E$ be the closure of the quadratic form $\E_0$ given by $\dom(\E_0)=\dom(\delta)$ and $\E_0(a)=\norm{\delta(a)}_\H^2$. Then the strongly continuous semigroup associated with $\E$ is the GNS implementation of a GNS-symmetric semigroup of contractive completely positive maps on the left von Neumann algebra generated by $\AA$.
\end{theorem*}
Here a normal Tomita bimodule is a bimodule over a Tomita algebra that additionally carries a complex one-parameter group $(\U_z)$ and an involution $\J$ satisfying some compatibility conditions, and a symmetric derivation $\delta\colon\AA\to\H$ is a map that intertwines the complex one-parameter groups and involutions on $\AA$ and $\H$ and satisfies the product rule
\begin{equation*}
\delta(ab)=a\delta(b)+\delta(a)b.
\end{equation*}
These objects were introduced in \cite{Wir22b} and appear to be the natural non-tracial analogs of the Hilbert bimodule and derivation occurring in the context of completely Dirichlet forms on tracial von Neumann algebras.

Combined with the results from \cite{Wir22b}, we thus obtain a comprehensive picture of GNS-symmetric quantum Markov semigroups analogous to the result of Cipriani and Sauvageot for tracially symmetric semigroups. Among other potential applications, we hope that this result opens the gate for applications to non-tracial free probability and deformation/rigidity theory of type III von Neumann algebras similar to recent work in this direction in the tracial case.

One main difficulty when trying to prove that closable derivations generate completely Dirichlet forms (or semigroups of completely positive maps) is that the property defining derivations, the product rule, is an algebraic property, while Dirichlet forms are defined in terms of order properties, and the domain of a derivation is not necessarily closed under order operations. As such, the problem of properly dealing with domains is crucial. Note that it is unavoidable to allow for unbounded derivations as everywhere defined derivations yield norm continuous semigroups of completely positive maps, which is too restrictive for many applications.

In the tracial case, this difficulty can be overcome since order operations such as taking the positive part can be expressed in terms of functional calculus and as such can be approximated by polynomials. In the non-tracial case, the order operations can still be expressed in terms of functional analysis in the setup of Haagerup $L^p$ spaces, but the product rule is formulated in terms of Hilbert algebra multiplication, which is different from the product of two operators in Haagerup $L^2$ (which is only in $L^2$ if it is zero). Therefore it is not clear how to connect the two.

Instead of trying to follow the proof in the tracial setting, our proof strategy instead relies on Haagerup reduction method, which allows to embed a von Neumann algebra as an expected subalgebra of a bigger von Neumann algebra that can be approximated by finite von Neumann algebras. As it turns out, this reduction method is well-suited to reduce the problem at hand to the known case of tracial von Neumann algebras. One key challenge are again domain issues: For the Haagerup construction one has to extend the derivation to a domain on a crossed product that is sufficiently big, but such that the extension still satisfies the product rule. The essential new technical ingredient to overcome this kind of domain problems lies in the introduction of a new locally convex topology on the domain of a derivation that allows to extend derivations to derivations on a completion.

As a final note, considering the results from \cite{VW23}, it is a natural question whether the results from the present article can be extended to cover KMS-symmetric semigroups. For one, our methods crucially use commutation with the modular group, which fails for KMS-symmetric maps if they are not GNS-symmetric. But more severely, it seems like there are additional algebraic obstructions, already in finite dimensions: It is shown in \cite{VW23} that if $\E$ is a completely Dirichlet form on $L^2(M_n(\IC),\phi)$, then there exist self-adjoint matrices $v_j\in M_n(\IC)$ such that
\begin{equation*}
\E(\rho^{1/4}x\rho^{1/4})=\sum_j \mathrm{tr}(\abs{\rho^{1/4}[v_j,x]\rho^{1/4}}^2),
\end{equation*}
where $\rho$ is the density matrix inducing the state $\phi$ on $M_n(\IC)$. However, without further assumptions on the operators $v_j$, the quadratic form on the right side of the previous equation is not necessarily a completely Dirichlet form.

\subsection*{Outline of the article}

In Section \ref{sec:basics} we recall some basics regarding modular theory, completely Dirichlet forms on standard forms of von Neumann algebras and Tomita bimodules and derivations. In Section \ref{sec:delta_top} we introduce a topology on the domain of a derivation, the $\delta$-topology, and show that derivations can be extended to derivations on the completion in the $\delta$-topology. In Section \ref{sec:closable} we give a closability criterion for derivations in our setting. In Section \ref{sec:proof_main} we discuss how derivations can be extended to crossed products and discuss how completely Dirichlet forms behave with respect to change of the reference weight. Then we state and prove the main result of this article, Theorem \ref{thm:main}, showing that the quadratic form associated with a closable derivation is a modular completely Dirichlet form. Finally, in Section \ref{sec:examples} we discuss several classes of examples, including inner derivations, derivations arising in non-tracial free probability and derivation induced by cocycles on (possibly non-unimodular) locally compact groups.

\subsection*{Acknowledgments} The author was funded by the Austrian Science Fund (FWF) under the Esprit Programme [ESP 156]. For the purpose of Open Access, the authors have applied a CC BY public copyright licence to any Author Accepted Manuscript (AAM) version arising from this submission.

\section{Basics}\label{sec:basics}

In this section we briefly recap some material concerning modular theory and in particular Hilbert and Tomita algebras, completely Dirichlet forms, Tomita bimodules and derivations that is used in the later sections.

\subsection{Modular theory}

As our approach is formulated in the language of Hilbert and Tomita algebras, we  summarize the relevant definitions here. Our treatment mostly follows \cite[Chapters VI--VII]{Tak03}.

An algebra $\AA$ with involution $^\sharp$ (resp. $^\flat$) and inner product $\langle\,\cdot\,,\cdot\,\rangle$ is called \emph{left (resp. right) Hilbert algebra} if
\begin{itemize}
\item for every $a\in\AA$ the map $\pi_l(a)\colon\AA\to \AA$, $b\mapsto ab$ (resp. $b\mapsto ba$) is bounded,
\item $\langle ab,c\rangle=\langle b,a^\sharp c\rangle$ (resp. $\langle ab,c\rangle=\langle b,ca^\flat\rangle$) for all $a,b,c\in\AA$,
\item the involution $^\sharp$ (resp. $^\flat$) is closable,
\item the linear span of all products $ab$ with $a,b\in\AA$ is dense in $\AA$.
\end{itemize}

Let $M$ be a von Neumann algebra and $\phi$ a normal semi-finite faithful weight on $M$. We write $\n_\phi$ for the definition ideal $\{x\in M\mid \phi(x^\ast x)<\infty\}$ and $(\pi_\phi,L^2(M,\phi),\Lambda_\phi)$ for the associated semi-cyclic representation.

The prototypical example of a left Hilbert algebra is $\AA=\Lambda_\phi(\n_\phi\cap\n_\phi^\ast)$ with the product $\Lambda_\phi(x)\Lambda_\phi(y)=\Lambda_\phi(xy)$, the involution $\Lambda_\phi(x)^\sharp=\Lambda_\phi(x^\ast)$ and the inner product inherited from $L_2(M,\phi)$, that is, $\langle\Lambda_\phi(x),\Lambda_\phi(y)\rangle=\phi(x^\ast y)$. In this case, $\pi_l(\AA)^{\prime\prime}=\pi_\phi(M)$. We write $\AA_\phi$ for this left Hilbert algebra.

Conversely, every left Hilbert algebra $\AA$ gives rise to a von Neumann algebra $\pi_l(\AA)^{\prime\prime}$ acting on the completion of $\AA$ and a weight 
\begin{equation*}
\phi\colon \pi_l(\AA)^{\prime\prime}_+\to [0,\infty],\,\phi(x)=\begin{cases}\norm{\xi}^2&\text{if }x^{1/2}=\pi_l(\xi),\\\infty&\text{otherwise}.\end{cases}
\end{equation*}
If $\AA$ is a full left Hilbert algebra \cite[Definition VI.1.16]{Tak03}, then $\phi$ is a normal semi-finite faithful weight on $\pi_l(\AA)^{\prime\prime}$, and $\AA$ is canonically isomorphic to $\AA_\phi$.

Let $\HH$ be the completion of the left Hilbert algebra $\AA$. Since the involution $^\sharp$ on $\AA$ is closable, its closure $S$ on $\HH$ exists and has a polar decomposition $S=J\Delta^{1/2}$. The operator $\Delta$ is a non-singular positive self-adjoint operator, called the \emph{modular operator}, and $J$ is an anti-unitary involution, called the \emph{modular conjugation}. If $\AA$ is the left Hilbert algebra associated with a weight $\phi$, we write $\Delta_\phi$ and $J_\phi$ for the associated modular operator and modular conjugation. We write $\Lambda_\phi^\prime\colon \n_\phi^\ast\to L_2(M,\phi)$ for the map $x\mapsto J_\phi\Lambda_\phi(x^\ast)$.

If $\AA$ is full, the modular conjugation $J$ gives rise to the positive self-dual cone $P=\overline{\{\pi_l(a)Ja\mid a\in\AA\}}$ and $\pi_l(\AA)^{\prime\prime}$ is in standard form \cite[Definition IX.1.13]{Tak03}.

The modular operator $\Delta$ gives rise to a point weak$^\ast$ continuous group of automorphisms $x\mapsto \Delta^{it}x\Delta^{-it}$ on $\pi_l(\AA)^{\prime\prime}$. If $\phi$ is a normal semi-finite faithful weight on $M$, the group $\sigma^\phi$ given by $\sigma^\phi_t(x)=\pi_\phi^{-1}(\Delta_\phi^{it}\pi_\phi(x)\Delta_\phi^{-it})$ is called the \emph{modular group} associated with $\phi$.

If $(\alpha_t)_{t\in\IR}$ is a point weak$^\ast$ continuous group of $\ast$-automorphisms on $M$, then an element $x\in M$ is called \emph{entire analytic} if the map $t\mapsto \alpha_t(x)$ has an extension $z\mapsto \alpha_z(x)$ to the complex plane such that $z\mapsto\omega(\alpha_z(x))$ is analytic for every $\omega\in M_\ast$. The entire analytic elements form a weak$^\ast$ dense $\ast$-subalgebra of $M$.

A \emph{Tomita algebra} is a left Hilbert algebra $\AA$ endowed with a complex one-parameter group $(U_z)_{z\in\IC}$ of algebra automorphism such that
\begin{itemize}
\item $z\mapsto \langle a,U_z b\rangle$ is analytic for all $a,b\in\AA$,
\item $(U_z a)^\sharp=U_{\bar z}(a^\sharp)$ for all $a\in\AA$, $z\in\IC$,
\item $\langle U_z a,b\rangle=\langle a,U_{-\bar z}b\rangle$ for all $a,b\in\AA$, $z\in \IC$,
\item $\langle a^\sharp,b^\sharp\rangle=\langle U_{-i}b,a\rangle$ for all $a,b\in \AA$.
\end{itemize}
Note that every Tomita algebra becomes a right Hilbert algebra when endowed with the involution
\begin{equation*}
\AA\to\AA,a\mapsto a^\flat=U_{-i}(a^\sharp).
\end{equation*}
For a full left Hilbert algebra $\AA$ let
\begin{equation*}
\AA_0=\left\{\xi\in\bigcap_{n\in\IZ}D(\Delta^n)\,\bigg\vert\, \Delta^n\xi\in\AA\text{ for all }n\in\IZ\right\}.
\end{equation*}
For every $\xi\in\AA_0$ the map $t\mapsto \Delta^{it}\xi$ has an entire analytic extension $z\mapsto U_z\xi$ with $U_z\xi\in \AA_0$ for all $z\in\IC$. This makes $\AA_0$ into a Tomita algebra such that $\pi_l(\AA_0)^{\prime\prime}=\pi_l(\AA)^{\prime\prime}$.

In particular,
\begin{equation*}
(\AA_\phi)_0=\{\Lambda_\phi(x)\mid x\in \n_\phi\cap\n_\phi^\ast,\,x\text{ entire analytic for }\sigma^\phi\}.
\end{equation*}

\subsection{Completely Dirichlet forms}

Completely Dirichlet forms in the non-tracial setting were introduced by Goldstein and Lindsay \cite{GL95,GL99} in the language of GNS Hilbert spaces of states (or weights) and by Cipriani \cite{Cip97} in the language of standard forms with a fixed cyclic vector. Our approach is somewhat different from both of these formulations in that we use left Hilbert algebras, but in view of the previous subsection it is equivalent to the formulation by Goldstein--Lindsay (and to that of Cipriani in case the left Hilbert algebra has a unit).

Let $\AA$ be a full left Hilbert algebra with completion $\HH$. Let $C$ be the closure of $\{\Delta^{1/4}a\mid a\in \AA,\,0\leq\pi_l(a)\leq 1\}$ and let $P_C$ be the metric projection onto $C$. We say that a closed densely defined quadratic form $\E$ on $\HH$ is a \emph{Dirichlet form} with respect to $\AA$ if $\E\circ J=\E$ and $\E(P_C(a))\leq \E(a)$ for all $a\in \HH$ with $Ja=a$.

The Dirichlet form $\E$ is called \emph{completely Dirichlet form} if for every $n\in\IN$ the quadratic form
\begin{equation*}
\E^{(n)}\colon \HH\otimes M_n(\IC)\to [0,\infty],\,\E^{(n)}([\xi_{ij}])=\sum_{i,j=1}^n \E(\xi_{ij})
\end{equation*}
is a Dirichlet form with respect to $\AA\odot M_n(\IC)$. Here $M_n(\IC)$ carries the normalized Hilbert--Schmidt inner product and the multiplication and involution on $\AA\odot M_n(\IC)$ are given by $[a_{ij}][b_{ij}]=[\sum_k a_{ik}b_{kj}]$, $[a_{ij}]^\sharp=[a_{ji}^\sharp]$.

A (completely) Dirichlet form with respect to $\AA$ is called \emph{modular} (or GNS-symmetric) if $\E\circ U_t=\E$ for all $t\in\IR$.

Completely Dirichlet forms are of particular interest for their connection to semigroups of contractive completely positive maps on von Neumann algebras. Let us briefly sketch this correspondence. Proofs can be found in \cite[Theorems 4.9, 5.7]{GL99} for the wider class of KMS-symmetric semigroups. The result for GNS-symmetric semigroups follows from the fact that GNS symmetry is equivalent to KMS symmetry and commutation with the modular group (see \cite[Proposition 6.1]{AC82} for example).

Let $M$ be a von Neumann algebra. A \emph{quantum dynamical semigroup} is a semigroup of normal contractive completely positive operators on $M$ that is continuous in the point weak$^\ast$ topology. If $\phi$ is a normal semi-finite faithful weight on $M$, a quantum dynamical semigroup $(P_t)$ is called \emph{GNS-symmetric} with respect to $\phi$ if $\phi\circ P_t\leq\phi$ for all $t\geq 0$ and
\begin{equation*}
\phi(P_t(x)^\ast y)=\phi(x^\ast P_t(y))
\end{equation*}
for all $x,y\in \n_\phi$ and $t\geq 0$.

Every GNS-symmetric quantum dynamical semigroup gives rise to a strongly continuous semigroup $(T_t)$ on $L^2(M,\phi)$, its \emph{GNS implementation}, acting by $T_t\Lambda_\phi(x)=\Lambda_\phi(P_t(x))$ for $x\in\n_\phi$, and the associated quadratic form is a modular completely Dirichlet form with respect to $\AA_\phi$. Vice versa, the strongly continuous semigroup associated with a modular completely Dirichlet form is the GNS implementation of a GNS-symmetric quantum dynamical semigroup.

We call a completely Dirichlet form a \emph{quantum Dirichlet form} if the associated quantum dynamical semigroups consists of unital maps. A criterion in terms of the form itself is given in \cite[Proposition 3.2]{Wir22b}.

\subsection{Tomita bimodules and derivations}

Tomita bimodules were introduced in \cite{Wir22b} as codomains of the derivations associated with modular completely Dirichlet forms.

Let $\AA$ be a Tomita algebra. A \emph{Tomita bimodule} over $\AA$ is an inner product space $\H$ endowed with non-degenerate commuting left and right actions of $\AA$, an anti-isometric involution $\J\colon\H\to\H$ and a complex one-parameter group $(\U_z)$ of isometries such that
\begin{itemize}
\item $\norm{a\xi b}\leq \norm{\pi_l(a)}\norm{\pi_r(b)}\norm{\xi}$ for $a,b\in\AA$, $\xi\in\H$,
\item $\langle a\xi b,\eta\rangle=\langle\xi, a^\sharp \eta b^\flat\rangle$ for $a,b\in\AA$, $\xi,\eta\in \H$,
\item $\U_z(a\xi b)=(U_z a)(\U_z \xi)(U_z b)$ for $a,b\in \AA$, $\xi\in\H$, $z\in\IC$,
\item $\J(a\xi b)=(Jb)(\J\xi)(Ja)$ for $a,b\in\AA$, $\xi\in\H$,
\item $\U_z \J=\J\U_{\bar z}$ for $z\in\IC$.
\end{itemize}
Let $\bar\H$ be the completion of $\H$. The first two bullet points imply that $\pi_l(a)\mapsto (\xi\mapsto a\xi)$ extends to a non-degenerate $\ast$-homorphism from $\pi_l(\AA)$ to $B(\bar\H)$. If this map can be extended to a normal $\ast$-homomorphism from $\pi_l(\AA)^{\prime\prime}$ to $B(\bar\H)$, then we say that $\H$ is a \emph{normal Tomita bimodule}. Requiring normality for the right action instead leads to the same notion of normal Tomita bimodule.

If $\AA$ is a Tomita algebra and $\H$ a bimodule over $\AA$, we call a linear map $\delta\colon \AA\to\H$ a \emph{derivation} if it satisfies the product rule
\begin{equation*}
\delta(ab)=a\delta(b)+\delta(a)b
\end{equation*}
for $a,b\in\AA$. If $\H$ is a Tomita bimodule over $\AA$, we say that a derivation $\delta\colon\AA\to\H$ is \emph{symmetric} if $\delta\circ J=\J\circ\delta$ and $\delta\circ U_z=\U_z\circ\delta$ for all $z\in\IC$.

If $\AA$ is a full left Hilbert algebra and $\E$ a modular quantum Dirichlet form with respect to $\AA$, it is shown in \cite[Theorem 6.3]{Wir22b} that
\begin{equation*}
\AA_\E=\{a\in \AA_0\mid U_z a\in \dom(\E)\text{ for all }z\in\IC\}
\end{equation*}
is a Tomita subalgebra of $\AA_0$ and a core for $\E$. Moreover, by \cite[Theorem 6.8]{Wir22b} there exists a Tomita bimodule $\H$ over $\AA$ and a symmetric derivation $\delta\colon \AA_\E\to\H$ such that
\begin{equation*}
\E(a,b)=\langle\delta(a),\delta(b)\rangle_\H
\end{equation*}
for $a,b\in \AA_\E$.

Under the minimality condition $\H=\mathrm{lin}\{\delta(a)b\mid a,b\in\AA_\E\}$, the pair $(\H,\delta)$ is uniquely determined by $\E$ up to isometric isomorphism preserving the Tomita bimodule structure and intertwining the derivations \cite[Theorem 6.9]{Wir22b}. By a slight abuse of notation, any such pair $(\H,\delta)$ is called the \emph{first-order differential structure associated with $\E$}. If $\H$ is a normal Tomita bimodule, the quantum Dirichlet form $\E$ is called \emph{$\Gamma$-regular}. A characterization in terms of the carré du champ is given in \cite[Theorem 7.2]{Wir22b}.

\section{\texorpdfstring{$\delta$-topology and completeness}{δ-topology and completeness}}\label{sec:delta_top}

In this section we introduce a locally convex topology on the domain of a closable symmetric derivation, called the $\delta$-topology. This topology is strong enough to ensure that the derivation extends to a derivation on the completion, which is a key technical ingredient in the proof of the main theorem later.

For the definition of the $\delta$-topology recall that the Mackey topology $\tau(M,M_\ast)$ on a von Neumann algebra $M$ is the finest linear topology $\mathcal T$ on $M$ such that the topological dual of $(M,\mathcal T)$ is $M_\ast$. Equivalently, it is the finest locally convex topology on $M$ that coincides with the strong$^\ast$ topology on norm bounded sets \cite[Proposition 2.11, Remark 2.12]{Coo87}. It has the advantage over the other usual locally convex topologies on $M$ of being complete, which is convenient for several of the following arguments.

Let $\AA$ be a Tomita algebra with completion $\HH$, let $\H$ be a normal Tomita bimodule over $\AA$ and $\delta\colon \AA\to\H$ a symmetric derivation. We define the $\delta$-topology $\mathcal T_\delta$ on $\AA$ as the coarsest locally convex topology that makes the maps
\begin{equation*}
\AA\to \HH\oplus\bar \H,\,a\mapsto(\Delta^n a,\delta(\Delta^n a))
\end{equation*}
continuous with respect to the norm topology on $\HH\oplus \bar \H$ for all $n\in \IZ$ and the maps
\begin{equation*}
\AA\to B(\HH),\,a\mapsto \pi_l(\Delta^n a)
\end{equation*}
continuous for the Mackey topology on $B(\HH)$ for all $n\in\IZ$. Clearly, the $\delta$-topology is stronger than the topology induced by the graph norm $(\norm{\cdot}_\HH^2+\norm{\delta(\cdot)}_\H^2)^{1/2}$.

\begin{lemma}
If $\E$ is a $\Gamma$-regular modular quantum Dirichlet form and $(\H,\delta)$ the associated first-order differential structure, then $\AA_\E$ is complete in the $\delta$-topology.
\end{lemma}
\begin{proof}
Let $(a_j)$ be a Cauchy net in $\AA_\E$ with respect to the $\delta$-topology. In particular, $(\Delta^n a_j,\delta(\Delta^n a_j))_j$ is Cauchy in $\HH\oplus\bar\H$ for all $n\in\IZ$. Since $\Delta^n$ and $\delta$ are closable on $\AA_\E$, it follows that there exists $a\in \HH$ such that $(\Delta^n a_j,\delta(\Delta^n a_j))\to (\Delta^n a,\bar\delta(\Delta^n a))$ for all $n\in\IZ$. In particular, $a\in\bigcap_{n\in\IZ}\dom(\Delta^n)$ and $\Delta^n a\in\dom(\bar \delta)=\dom(\E)$ for all $n\in\IZ$.

Moreover, as the Mackey topology is complete, for $n\in\IZ$ there exists $x_n\in B(\HH)$ such that $\pi_l(\Delta^n a_j)\to x_n$ with respect to $\tau(B(\HH),B(\HH)_\ast)$. For $b\in\AA_\E$ we have
\begin{equation*}
x_n b=\lim_j \pi_l(\Delta^n a_j)b=\lim_j \pi_r(b)\Delta^n a_j=\pi_r(b)\Delta^n a.
\end{equation*}
Since $\AA_\E$ is dense in $\HH$, it follows that $\Delta^n a\in \AA_\E^{\prime\prime}$ and $\pi_l(\Delta^n a_j)\to\pi_l(\Delta^n a)$ for all $n\in\IZ$.

Altogether we conclude that $a\in \AA_\E$ and $a_j\to a$ in the $\delta$-topology.
\end{proof}

\begin{lemma}\label{lem:diff_calc_delta_complete}
If $\AA$ is a Tomita algebra, $\H$ a Tomita bimodule over $\AA$ and $\delta\colon \AA\to \H$ a closable symmetric derivation, the inclusion of $\AA$ into its completion $\HH$ extends to an injective map from the completion of $\AA$ in the $\delta$-topology to $\HH$.
\end{lemma}
\begin{proof}
We have to show that if $(a_j)$ is a Cauchy net in $\AA$ with respect to the $\delta$-topology and $a_j\to 0$ in $\HH$, then $a_j\to 0$ in the $\delta$-topology. Since $\Delta^n$, $n\in\IZ$, and $\delta$ are closable, we have $(\Delta^n a_j,\delta(\Delta^n a_j))\to 0$ for all $n\in\IZ$. Furthermore, using the completeness of the Mackey topology and a similar argument as in the previous lemma, one sees that $\pi_l(\Delta^n a_j)\to 0$ in $\tau(B(\HH),B(\HH)_\ast)$ for all $n\in\IZ$. Hence $a_j\to 0$ in the $\delta$-topology.
\end{proof}

Let $\widehat{\AA}^\delta$ denote the set of all elements $a\in\HH$ for which there exists a net $(a_j)$ in $\AA$ such that $a_j\to a$ in $\HH$ and $(a_j)$ is Cauchy in the $\delta$-topology. By the previous lemma, $\widehat{\AA}^\delta$ is a completion of $\AA$ in the $\delta$-topology, and we call it simply the $\delta$-completion of $\AA$. It is not hard to see that $\widehat{\AA}^\delta$ is a Tomita subalgebra of $(\AA^{\prime\prime})_0$ and contained in $\dom(\bar\delta)$.

Recall that if $\HH$ is a normal Tomita bimodule over $\AA$, we can continuously extend the left and right action of $\AA$ and the maps $\J$ and $\U_t$, $t\in\IR$, to the Hilbert completion $\bar\H$. This is usually not possible for $\U_z$, $z\in\IC\setminus\IR$. We define
\begin{equation*}
\H^{\mathrm{a}}=\{\xi\in \bar\H\mid t\mapsto \bar\U_t\xi\text{ has an entire extension}\}.
\end{equation*}
If it exists, this entire extension is unique and will be denoted by $z\mapsto \bar\U_z\xi$. Clearly, $\H\subset\H^{\mathrm{a}}$ and $\U_z\subset \bar\U_z$ for all $z\in\IC$.

\begin{lemma}
If we endow $\H^{\mathrm{a}}$ with the coarsest locally convex topology that makes
\begin{equation*}
\H^{\mathrm{a}}\to\bar\H,\,\xi\mapsto\bar\U_{in}\xi
\end{equation*}
continuous for all $n\in\IZ$, then $\H^{\mathrm{a}}$ is complete.
\end{lemma}
\begin{proof}
If $A$ is the unique non-singular positive self-adjoint operator in $\bar\H$ such that $\bar\U_t=A^{it}$ for $t\in\IR$, then $\H^{\mathrm{a}}=\bigcap_{n\in\IZ}\dom(A^n)$ and $\H^{\mathrm{a}}$ is the projective limit of the Banach spaces $(\dom(A^n)\cap \dom(A^{-n}),\norm{\cdot}_{\bar\H}+\norm{A^n \,\cdot\,}_{\bar\H}+\norm{A^{-n}\,\cdot\,}_{\bar\H})$ in the topology described in the lemma. In particular, $\H^{\mathrm{a}}$ is complete.
\end{proof}


Since $\H$ is a \emph{normal} Tomita bimodule over $\AA$, the Hilbert completion $\bar\H$ has a canonical structure of a $\pi_l(\AA)^{\prime\prime}$-$\pi_l(\AA)^{\prime\prime}$ correspondence determined by
\begin{equation*}
\pi_l(a)\cdot \xi\cdot J\pi_r(b)^\ast J=a\xi b
\end{equation*}
for $a,b\in \AA$ and $\xi\in \H$.

If $a\in \widehat{\AA}^\delta$, $(a_j)$ is a net in $\AA$ such that $a_j\to a$ in the $\delta$-topology and $\xi\in\bar\H$, then
\begin{equation*}
\bar\U_t(\pi_l(a)\cdot \xi)=\lim_j \bar\U_t(\pi_l(a_j)\xi)=\lim_j \pi_l(U_t a_j)\bar\U_t\xi=\pi_l(U_t a)\cdot\bar\U_t\xi.
\end{equation*}
Thus, if $\xi\in\H^{\mathrm{a}}$, then $z\mapsto \pi_l(U_z a)\cdot \bar\U_z\xi$ is an entire continuation of $t\mapsto \bar\U_t(\pi_l(a)\cdot\xi)$, which implies $\pi_l(a)\xi\in \H^{\mathrm{a}}$. Likewise, if $b\in \widehat{\AA}^\delta$, then $\xi\cdot J\pi_r(b)^\ast J\in \H^{\mathrm{a}}$. It is then routine to check that the bimodule structure given by $a\xi b=\pi_l(a)\cdot\xi\cdot J\pi_r(b)^\ast J$, then group $(\bar\U_z)_{z\in\IC}$ and the restriction of $\bar\J$ make $\H^{\mathrm{a}}$ into a Tomita bimodule over $\widehat{\AA}^\delta$.

\begin{lemma}\label{lem:completion_deriv}
If $\AA$ is a Tomita algebra, $\H$ is a Tomita bimodule over $\AA$ and $\delta\colon\AA\to\H$ is a closable symmetric derivation with closure $\bar\delta$, then $\bar\delta(\widehat{\AA}^\delta)\subset \H^{\mathrm{a}}$ and $\bar\delta\colon \widehat{\AA}^\delta\to\H^{\mathrm{a}}$ is a symmetric derivation.
\end{lemma}
\begin{proof}
If $a\in\widehat\AA^\delta$ and $(a_j)$ is a net in $\AA$ such that $a_j\to a$ in the $\delta$-topology, then
\begin{equation*}
\bar\U_t\bar\delta(a)=\lim_j \U_t \delta(a_j)=\delta(U_t a_j)=\bar\delta(U_t a).
\end{equation*}
It follows that $t\mapsto \bar\U_t\bar\delta(a)$ has the entire continuation $z\mapsto \bar\delta(U_z a)$, which implies $\bar\delta(a)\in\bar\H^{\mathrm{a}}$. Again, routine computations show that the restriction of $\bar\delta$ to $\widehat{\AA}^\delta$ is a symmetric derivation from $\widehat\AA^\delta$ to $\H^{\mathrm{a}}$.
\end{proof}

\section{Closability of derivations}\label{sec:closable}

In this section we give a simple criterion for the closability of derivations inspired by a well-known result (see \cite[Section 4]{Voi98} and \cite[Lemma 3.9]{Nel17} for the non-tracial case) on the closability of the derivation used in free probability.

If $\AA$ is a Tomita algebra and $\H$ is a Tomita bimodule over $\AA$, we say that $\xi\in\H$ is a \emph{bounded vector} if there exists $C>0$ such that $\norm{a\xi b}\leq C\norm{a}\norm{b}$ for all $a,b\in \AA$. In this case, the maps $a\mapsto a\xi$ and $b\mapsto \xi b$ extend to bounded linear operators from the completion $\HH$ of $\AA$ to $\H$, which we denote by $R(\xi)$ and $L(\xi)$, respectively.

\begin{lemma}
Let $\AA$ be a Tomita algebra, $\H$ a normal Tomita bimodule over $\AA$ and $\delta\colon\AA\to\H$ a derivation. If $\delta(\AA)$ is contained in the space of bounded vectors, then $\dom(\delta^\ast)$ is a subbimodule of $\H$ and
\begin{equation*}
\delta^\ast(a\xi b)=a^\ast\delta^\ast(\xi)b-L(\delta(a^\ast))^\ast(\xi b)-R(\delta(b^\ast))^\ast (a\xi)
\end{equation*}
for $a,b\in\AA$ and $\xi\in\dom(\delta^\ast)$.
\end{lemma}
\begin{proof}
Let $a,b,c\in \AA$ and $\xi\in\dom(\delta^\ast)$. By the product rule,
\begin{align*}
\langle a\xi b,\delta(c)\rangle&=\langle \xi,a^\ast\delta(c)b^\ast\rangle\\
&=\langle \xi,\delta(a^\ast c b^\ast)-\delta(a^\ast)cb^\ast-a^\ast c\delta(b^\ast)\rangle\\
&=\langle a\delta^\ast(\xi)b-L(\delta(a^\ast))^\ast(\xi b)-R(\delta(b^\ast))^\ast (a\xi),c\rangle.
\end{align*}
Thus $a\xi b\in\dom(\delta^\ast)$ and the claimed identity for $\delta^\ast(a\xi b)$ holds.
\end{proof}

\begin{lemma}
Let $\AA$ be a Tomita algebra, $\H$ a normal Tomita bimodule over $\H$ and $\delta\colon\AA\to\H$ a derivation. If $\delta(\AA)$ is contained in the space of bounded vectors and $\dom(\delta^\ast)$ is a cyclic subset, then $\delta$ is closable.
\end{lemma}
\begin{proof}
By the previous lemma, $\dom(\delta^\ast)$ is a subbimodule of $\H$. Hence, if $\dom(\delta^\ast)$ is cyclic, then it is dense in $\H$. Therefore, $\delta$ is closable.
\end{proof}

\section{Completely Dirichlet forms associated with closable derivations}\label{sec:proof_main}

In this section we prove the main theorem of this article, Theorem \ref{thm:main}, showing that the closure of the quadratic form associated with a closable symmetric derivation is a modular completely Dirichlet form.

As mentioned in the introduction, we rely on Haagerup's reduction method. To set up the stage for its use, we first discuss crossed products of Tomita algebras and Tomita bimodules. To extend closable symmetric derivations to a sufficiently large domains on the crossed product, we use the $\delta$-completion technique developed in Section \ref{sec:delta_top}. Further, to reduce the problem to the tracial case, we need a ``change of reference weight'' argument and an analysis of approximation properties of completely Dirichlet forms. This will be dealt with in the following lemmas. Finally, in Proposition \ref{prop:rel_deriv_diff_calc} we discuss the relation between the derivation we started with and the first-order differential structure of the associated completely Dirichlet form.

Let $\AA$ be a Tomita algebra, $\H$ a normal Tomita bimodule over $\AA$ and $\delta\colon \AA\to \H$ a closable symmetric derivation. Throughout this section we endow $\AA$ with the $\delta$-topology and $\H$ with the projective topology induced by the maps $\H\to\bar\H,\,\xi\mapsto\bar\U_{in}\xi$ for $n\in \IZ$, and we assume that $\AA$ and $\H$ are complete in these topologies. As discussed in Section \ref{sec:delta_top}, this can always be achieved by passing to the completions.

Let $G$ be a countable subgroup of $\IR$, viewed as discrete group. The vector space $C_c(G;\AA)\cong C_c(G)\odot \AA$ can be made into a Tomita algebra by the operations
\begin{align*}
(a\ast b)(g)&=\sum_{h\in G}U_{-h}a(g-h)b(h),\\
a^\sharp(g)&=U_{-g}(a(-g)^\sharp),\\
(U_z a)(g)&=U_z a(g).
\end{align*}
Moreover, the vector space $C_c(G;\H)$ becomes a normal Tomita bimodule over $C_c(G;\AA)$ with the operations 
\begin{align*}
(a\xi)(g)&=\sum_{h\in G}U_{-h} a(g-h)\xi(h)\\
(\xi b)(g)&=\sum_{h\in G}\U_{-h}\xi(g-h)b(h)\\
(\tilde\J\xi)(g)&=\U_{-g}\J\xi(-g)\\
(\tilde \U_z \xi)(g)&=\U_z \xi(g).
\end{align*}
Furthermore, $1_{C_c(G)}\odot \delta\colon C_c(G;\AA)\to C_c(G;\H)$ is a closable symmetric derivation, whose closure we denote by $1\otimes\bar\delta$.

We write $\tilde \AA$ for the $(1\odot\delta)$-completion of $C_c(G;\AA)$, $\tilde\H$ for $C_c(G;\H)^{\mathrm{a}}$ and $\tilde\delta$ for the restriction of $1\otimes\bar\delta$ to $\tilde\AA$. By Lemma \ref{lem:completion_deriv} the map $\tilde\delta$ is a (closable) symmetric derivation from $\tilde\AA$ to $\tilde\H$.

\begin{lemma}\label{lem:completion_crossed_prod}
If $x\in L(G)\otimes \IC 1_{\bar\H}$ and $a\in \tilde\AA$, then $x a, a x\in \tilde\AA$ and $\tilde\delta(x a)=x\tilde\delta(a)$, $\tilde\delta(a x)=\tilde\delta(a)x$.
\end{lemma}
\begin{proof}
Let $x=y\otimes 1$ with $y\in L(G)$, let $(y_i)$ be a bounded net in $\IC[G]$ such that $y_i\to y$ in the strong$^\ast$ topology and let $x_i=y_i\otimes 1$. Clearly, $x_i\to x$ in the Mackey topology.

If $a\in C_c(G;\AA)$, then $x_i a\in C_c(G;\AA)$ and $\Delta^n(x_i a)=x_i \Delta^n a$, $(1\odot\delta)(x_i a)=x_i(1\odot \delta)(a)$, $\pi_l(\Delta^n (x_i a))=x_i\pi_l(\Delta^n a)$. It follows that $x_i a\to xa$ in $\ell^2(G;\HH)$, the net $(x_i a)$ is Cauchy in the $\tilde\delta$-topology and $(1\odot\delta)(x_i a)\to x(1\odot \delta)(a)$. Thus $xa\in \tilde\AA$ and $\tilde\delta(xa)=x\tilde\delta(a)$.

A similar argument shows that if $(a_j)$ is a Cauchy net in $C_c(G;\AA)$ with respect to $\mathcal T_{\tilde\delta}$ and $a_j\to a$ in $\ell^2(G;\H)$, then $(x a_j)$ is Cauchy with respect to $\mathcal T_{\tilde\delta}$ and $xa_j\to xa$ in $\ell^2(G;\H)$. Hence if $a\in \tilde\AA$, then $xa\in \tilde\AA$ and $\tilde\delta(xa)=x\tilde\delta(a)$. The statement for $ax$ can be proven analogously.
\end{proof}

For the next lemma recall that $\AA_\phi=\Lambda_\phi(\n_\phi\cap \n_\phi^\ast)$ is the full left Hilbert algebra induced by the weight $\phi$, the cone $C_\phi$ is the closure of $\{\Delta_\phi^{1/4}a\mid a\in\AA_\phi,0\leq \pi_l(a)\leq 1\}$, and $M_\phi$ denotes the centralizer of $\phi$.

\begin{lemma}\label{lem:change_weight}
Let $M$ be a von Neumann algebra, $\phi$ a normal semi-finite faithful weight on $M$ and $x\in M_\phi$ be positive and invertible. Let $\psi=\phi(x^{1/2}\cdot x^{1/2})$.

If $\E$ is a modular (completely) Dirichlet form on $L^2(M)$ with respect to $\AA_\phi$, $x\dom(\E)\subset\dom(\E)$ and $\E(xa,b)=\E(a,xb)$ for all $a,b\in\dom(\E)$, then $\E$ is also a modular (completely) Dirichlet from with respect to $\AA_{\psi}$.
\end{lemma}
\begin{proof}
Since $x$ is invertible, the weight $\psi$ is faithful and $J_\psi=J_\phi$, and since $x$ commutes with $(\Delta_\phi^{it})$, we have $\Delta_{\psi}^{it}=x^{it}\Delta_\phi^{it}(\cdot)x^{-it}=\Delta_\phi^{it}(x^{it}\cdot x^{-it})$.

Let $A$ be the positive self-adjoint operator associated with $\E$ and $T_t=e^{-tA}$. The commutation relation $x\dom(\E)\subset \dom(\E)$ and $\E(xa,b)=\E(a,xb)$ for $a,b\in \dom(\E)$ implies that $x$ commutes strongly with $A^{1/2}$. Hence $T_t(xa)=x T_t(a)$ for all $a\in H$ and $t\geq 0$. Since $T_t$ commutes with $J_\phi$, we also have $T_t(ax)=T_t(a)x$ for $a\in H$, $t\geq 0$. In particular, $(T_t)$ and $(\Delta_{\psi}^{is})$ commute.

Moreover, $C_\psi=x^{1/4}C_\phi x^{1/4}$. Indeed, a direct computation shows that $\n_\psi=\n_\phi$ and $\Lambda_\psi(y)=\Lambda_\phi(y)x^{1/2}$ for $y\in \n_\phi$. Hence, if $y\in\n_\phi$ with $0\leq y\leq 1$, then 
\begin{equation*}
\Delta_\psi^{1/4}\Lambda_\psi(y)=x^{1/4}\Delta_\phi^{1/4}\Delta_\phi(y)x^{1/4}\in x^{1/4}C_\phi x^{1/4}.
\end{equation*}
The converse inclusion follows by swapping the roles of $\phi$ and $\psi$.

Therefore, if $a\in C_\psi$, then 
\begin{equation*}
T_t(a)=x^{1/4}T_t(x^{-1/4}a x^{-1/4})x^{1/4}\in C_\psi.
\end{equation*}
Thus $\E$ is a Dirichlet form with respect to $\AA_\psi$ by \cite[Theorem 5.7]{GL99}. The result for completely Dirichlet forms follows easily by applying the same argument to the forms $\E^{(n)}$ on $L^2(M\otimes M_n(\IC))$.
\end{proof}

\begin{lemma}\label{lem:Dirichlet_approx}
Let $M$ be a von Neumann algebra and $\phi$ a normal semi-finite faithful weight on $M$. Let $(M_n)$ be an increasing sequence of von Neumann subalgebras with weak$^\ast$ dense union and assume that $M_n$ is the range of a $\phi$-preserving conditional expectation $E_n$ on $M$. Let $H_n$ denote the closure of $\Lambda_{\phi}(\n_{\phi}\cap M_n)$ and let $P_n$ denote the orthogonal projection from $H$ to $H_n$.

If $\E$ is a closed densely defined quadratic form on $H$ such that for every $n\in \IN$ the quadratic form $\E\vert_{H_n}$ is a Dirichlet form with respect to $\Lambda_\phi(\n_\phi\cap \n_\phi^\ast\cap M_n)$ and $\E\circ P_n\leq\E$, then $\E$ is a Dirichlet form with respect to $\AA_\phi$.
\end{lemma}
\begin{proof}
Let $(T_t)$ be the strongly continuous semigroup associated with $\E$. Since $\E\circ P_n\leq \E$, we have $T_t(H_n)\subset H_n$ by Ouhabaz' theorem \cite[Corollary 2.4]{Ouh96}. Thus $T_t$ commutes with $P_n$, and it is easy to see that $(T_t P_n)$ is the semigroup associated with $\E\vert_{H_n}$, viewed as semigroup on $H$. In particular, $T_t P_n J_\phi=J_\phi T_t P_n$. In the limit we obtain $T_t J_\phi=J_\phi T_t$.

It remains to show that $T_t(C_\phi)\subset C_\phi$ for all $t\geq 0$. A direct computation shows that $E_n$ and $P_n$ are related by $P_n\Lambda_\phi(x)=\Lambda_\phi(E_n(x))$. Moreover, $E_n$ is GNS-symmetric with respect to $\phi$, which implies that $P_n$ commutes with $(\Delta_\phi^{it})$. Thus $P_n(C_\phi)$ is the closure of $\{\Delta_\phi^{1/4}\Lambda_\phi(x)\mid x\in \n_\phi\cap\n_\phi^\ast\cap M_n,\,0\leq x\leq 1\}$. In particular, $P_n(C_\phi)\subset C_\phi$.

Since $\E_n$ is a Dirichlet form with respect to $\Lambda_\phi(\n_\phi\cap\n_\phi^\ast\cap M_n)$ and $(T_t P_n)$ is the associated semigroup, we have $T_t P_n(C_\phi)\subset P_n (C_\phi)$. Moreover, since $\bigcup_n M_n$ is weak$^\ast$ dense in $M$, we have $P_n\to 1$ strongly by Kaplansky's density theorem. Therefore, $T_t(C_\phi)\subset C_\phi$.
\end{proof}

To prove that the quadratic form associated with a closable symmetric derivation is a completely Dirichlet form, we will reduce the problem to the tracially symmetric case by means of Haagerup's reduction method. We only recall the necessary definitions here and refer to \cite{HJX10} for proofs in the case of states and to \cite{CPPR15} for the extension to weights.

Let $M=\pi_l(\AA)^{\prime\prime}$, let $\phi$ be the weight induced by the full left Hilbert algebra $\AA^{\prime\prime}$ on $M$, let $G=\bigcup_{n\in\IN}2^{-n}\IZ$, let  $\tilde M=M\rtimes_{\sigma^\phi}G=\pi_l(\tilde\AA)^{\prime\prime}$ and let $\tilde\phi$ be the dual weight of $\phi$ on $\tilde M$. Let $(a_n)$ be a sequence of self-adjoint elements of $L(G)\otimes\IC 1\subset\mathcal Z(\tilde M_{\tilde\phi})$, $\phi_n=\phi e^{-a_n}$, $M_n=\tilde M_{\phi_n}$ and $\tau_n=\phi_n|_{M_n}$. Here $N_\psi$ denotes the centralizer of the weight $\psi$ on $N$ and $\mathcal Z(M)$ is the center of the von Neumann algebra $N$.

By \cite[Theorem 8.1]{CPPR15} the sequence $(a_n)$ can be chosen such that
\begin{itemize}
\item $M_n$ is semi-finite with normal semi-finite faithful trace $\tau_n$,
\item for each $n\in \IN$ there exists a conditional expectation $E_n$ from $M$ onto $M_n$ such that $\tilde\phi\circ E_n=\tilde\phi$ and $\sigma^{\tilde\phi}_t\circ E_n=E_n\circ\sigma^{\tilde\phi}_t$ for all $t\in\IR$,
\item $E_n(x)\to x$ strongly$^\ast$ for every $x\in M$.
\end{itemize}
In the following we fix a sequence $(a_n)$ with these properties. The concrete construction is irrelevant for our purposes.

\begin{theorem}\label{thm:main}
Let $\AA$ be a Tomita algebra with completion $\HH$, let $\H$ be a normal Tomita bimodule over $\AA$ and $\delta\colon \AA\to \H$ a closable symmetric derivation. The closure $\E$ of the quadratic form
\begin{equation*}
\HH\to [0,\infty],\,a\mapsto \begin{cases}\norm{\delta(a)}_\H^2&\text{if }a\in\AA,\\\infty&\text{otherwise}\end{cases}
\end{equation*}
is a modular completely Dirichlet form with respect to $\AA^{\prime\prime}$. If moreover $\AA$ is unital, then $\E$ is a modular quantum Dirichlet form.
\end{theorem}
\begin{proof}
We continue to use the notation from the previous discussion. The derivation $\tilde\delta\colon\tilde\AA\to\tilde\H$ is a restriction of $1\otimes \bar\delta$. Let $\tilde \E$ denote the closure of the quadratic form
\begin{equation*}
\ell^2(G;\HH)\to [0,\infty],\,a\mapsto\begin{cases}\norm{\tilde\delta(a)}_{\tilde\H}^2&\text{if }a\in \tilde\AA,\\\infty&\text{otherwise}.\end{cases}
\end{equation*}
It is clear that $\tilde \E(a)=\norm{(1\otimes\bar\delta)(a)}_{\ell^2(G;\bar\H)}^2$ for $a\in \dom(\tilde\E)$. Furthermore, $\dom(\tilde\E)=\dom(1\otimes\bar\delta)$ and the strongly continuous semigroups $(T_t)$ and $(\tilde T_t)$ associated with $\E$ and $\tilde \E$, respectively, are related by $\tilde T_t=\id_{\ell^2(G)}\otimes T_t$.

The map $\iota\colon \HH\to \ell^2(G;\HH),\,a\mapsto \1_0\otimes a$ is an isometric embedding such that $\iota(C_{\AA^{\prime\prime}})= C_{\tilde \AA^{\prime\prime}}\cap \iota(\HH)$. Thus, if $\tilde\E$ is a (completely) Dirichlet form with respect to $\tilde \AA^{\prime\prime}$, then $\E$ is a (completely) Dirichlet form with respect to $\AA^{\prime\prime}$.

Since $M$ is in standard form on $\HH$ and $\tilde M$ is in standard form on $\ell^2(G;\HH)$, these spaces can be canonically identified with $L^2(M,\phi)$ and $L^2(\tilde M,\tilde\phi)$, respectively, and we will tacitly do so in the following. Under these identifications, $\AA^{\prime\prime}=\AA_\phi$, $\tilde\AA^{\prime\prime}=\AA_{\tilde\phi}$ and $\Delta_{\tilde\phi}^{it}=\id_{\ell^2(G)}\otimes \Delta_\phi^{it}$.

Let
\begin{equation*}
\mathcal A_n=\{x\in \n_{\tilde\phi}\cap \n_{\tilde\phi}^\ast\cap M_n\mid \Lambda_{\tilde\phi}(x e^{-a_n/2})\in \tilde\AA\}.
\end{equation*}
Since $e^{a_n/2}\in \tilde M_{\tilde\phi}$, if $x\in\mathcal A_n$, then
\begin{equation*}
\Lambda_{\tilde\phi}(x)=\Lambda_{\tilde\phi}(x e^{-a_n/2})e^{a_n/2}\in \tilde\AA
\end{equation*}
by Lemma \ref{lem:completion_crossed_prod}. Reversing the roles of $e^{-a_n/2}$ and $e^{a_n/2}$ we get
\begin{equation*}
\mathcal A_n=\{x\in \n_{\tilde\phi}\cap\n_{\tilde\phi}^\ast\cap M_n\mid \Lambda_{\tilde\phi}(x)\in\tilde\AA\}.
\end{equation*}
Since $\tilde \AA$ is a Tomita algebra, it follows easily that $\mathcal A_n$ is a $\ast$-algebra. Define an $\mathcal A_n$-$\mathcal A_n$-bimodule structure on $\ell^2(G;\bar\H)$ by
\begin{equation*}
x\xi y=\Lambda_{\tilde\phi}(x)\cdot \xi\cdot \Lambda_{\tilde\phi}^\prime(y).
\end{equation*}
Using that $\tilde\H$ is a Tomita bimodule over $\tilde\AA$, it is not hard to see that this left and right action are contractive (anti-) $\ast$-homomorphisms. Moreover, $\tilde\J$ extends to an anti-unitary involution on $\ell^2(G;\bar\H)$ intertwining the left and right action. We still denote this extension by $\tilde\J$.

Let
\begin{equation*}
\partial_n\colon \mathcal A_n\to L^2(\tilde M,\tilde\phi),\,\partial_n(x)=\tilde \delta(\Lambda_{\tilde\phi}(xe^{-a_n/2})).
\end{equation*}
Since $e^{-a_n/2}\in \tilde M_{\tilde\phi}$ and $x\in \tilde M_{\phi_n}$, we have
\begin{align*}
\partial_n(x^\ast)&=\tilde\delta(\Lambda_{\tilde\phi}(x^\ast e^{-a_n/2}))\\
&=\tilde\delta(\Lambda_{\tilde\phi}^\prime( e^{-a_n/2}x^\ast))\\
&=\tilde\delta(\tilde J\Lambda_{\tilde\phi}(xe^{-a_n/2}))\\
&=\tilde\J\tilde\delta(\Lambda_{\tilde\phi}(xe^{-a_n/2}))\\
&=\tilde\J\partial_n(x).
\end{align*}
Moreover, it follows from Lemma \ref{lem:completion_crossed_prod} combined with $e^{-a_n/2}\in \tilde M_{\tilde\phi}$ and $x,y\in \tilde M_{\phi_n}$ that
\begin{align*}
\partial_n(xy)&=\tilde\delta(\Lambda_{\tilde\phi}(xye^{-a_n/2}))\\
&=\Lambda_{\tilde\phi}(x)\cdot \tilde\delta(\Lambda_{\tilde\phi}(ye^{-a_n/2}))+\tilde\delta(\Lambda_{\tilde\phi}(x))\cdot \Lambda_{\tilde\phi}(y e^{-a_n/2})\\
&=\Lambda_{\tilde\phi}(x)\cdot \tilde\delta(\Lambda_{\tilde\phi}(ye^{-a_n/2}))+\tilde\delta(\Lambda_{\tilde\phi}(xe^{-a_n/2}))\cdot \Lambda_{\tilde\phi}(e^{a_n/2} y e^{-a_n/2})\\
&=\Lambda_{\tilde\phi}(x)\cdot \tilde\delta(\Lambda_{\tilde\phi}(ye^{-a_n/2}))+\tilde\delta(\Lambda_{\tilde\phi}(xe^{-a_n/2}))\cdot \Lambda_{\tilde\phi}^\prime(\sigma^{\phi_n}_{-i/2}(y))\\
&=\Lambda_{\tilde\phi}(x)\cdot \tilde\delta(\Lambda_{\tilde\phi}(ye^{-a_n/2}))+\tilde\delta(\Lambda_{\tilde\phi}(xe^{-a_n/2}))\cdot \Lambda_{\tilde\phi}^\prime(y)\\
&=x\partial_n(y)+\partial_n(x)y.
\end{align*}
The operator $\partial_n$ is closable when viewed as operator in $L^2(M_n,\tau_n)$ since $\tilde\delta$ is closable and the map $\Lambda_{\tau_n}(x)\mapsto \Lambda_{\tilde\phi}(xe^{-a_n/2})$ extends to an isometry $\iota_n$ from $L^2(M_n,\tau_n)$ to $L^2(\tilde M,\tilde\phi)$.

Since $\tau_n$ is a trace, \cite[Theorem 4.5]{Cip08} implies that the closure $Q_n$ of the quadratic form
\begin{equation*}
L^2(M_n,\tau_n)\to [0,\infty],\,a\mapsto \begin{cases}\norm{\partial_n(x)}^2&\text{if }a=\Lambda_{\tau_n}(x),\,x\in \mathcal A_n,\\\infty&\text{otherwise}\end{cases}
\end{equation*}
is a completely Dirichlet form.

Let $H_n=\overline{\Lambda_{\tilde\phi}(\n_{\tilde\phi}\cap\n_{\tilde\phi}^\ast\cap M_n)}$ and let $\E_n$ be the closure of the quadratic form
\begin{equation*}
H_n\to [0,\infty],\,a\mapsto\begin{cases}\norm{\tilde\delta(a)}^2&\text{if }a\in \tilde\AA,\\\infty&\text{otherwise}.\end{cases}
\end{equation*}
In other words, $\E_n=Q_n\circ \iota_n^{-1}$.

Note that $\iota_n$ maps $\{\Lambda_{\tau_n}(x)\mid x\in \n_{\tilde\phi}\cap\n_{\tilde\phi}^\ast\cap M_n, 0\leq x\leq 1\}$ onto $\{\Lambda_{\tilde\phi}(x e^{-a_n/2})\mid x\in  \n_{\tilde\phi}\cap\n_{\tilde\phi}^\ast\cap M_n, 0\leq x\leq 1\}$. Since $\phi_n$ is a trace on $M_n$, the latter set coincides with $\{\Lambda_{\phi_n}(x)\mid x\in \AA_{\phi_n}\cap H_n,\,0\leq x\leq 1\}$. It follows that $\E_n$ is a completely Dirichlet form with respect to $\AA_{\phi_n}\cap H_n$.

Moreover,
\begin{align*}
\E_n(\Delta_{\phi_n}^{it} a)&=\E_n(e^{-i a_n/2 t}(\Delta_{\tilde\phi}^{it} a)e^{ia_n/2 t})\\
&=\norm{e^{-ia_n/2 t}(\tilde\U_t\tilde\delta(a))e^{ia_n/2 t}}^2\\
&=\norm{\tilde\delta(a)}^2\\
&=\E_n(a)
\end{align*}
for $a\in \tilde\AA$. This can easily be extended to the closure so that $\E_n$ is a modular completely Dirichlet form.

By Lemma \ref{lem:completion_crossed_prod} we have $e^{-a_n/2}\dom(\E_n)\subset \dom(\E_n)$ and $\E_n(e^{a_n/2}a,b)=\E_n(a,e^{-a_n/2}b)$ for $a,b\in \dom(\E_n)$. Furthermore, $e^{-a_n/2}\in \tilde M_{\tilde\phi}$. Hence $\E_n$ is also a modular completely Dirichlet form with respect to $\AA_{\tilde\phi}\cap H_n=\Lambda_{\tilde\phi}(\n_{\tilde\phi}\cap \n_{\tilde\phi}^\ast\cap M_n)$ by Lemma \ref{lem:change_weight}. 

Let $P_n$ denote the orthogonal projection from $\ell^2(G;\HH)$ onto $H_n$. By definition, $\tilde\E|_{H_n}=\E_n$. To apply Lemma \ref{lem:Dirichlet_approx}, we have to check that $\E\circ P_n\leq \E$.

Let $(T_t)$ be the strongly continuous semigroup associated with $\E$. As discussed above, $(\id_{\ell^2(G)}\otimes T_t)$ is the strongly continuous semigroup associated with $\tilde \E$. The modular group of $\phi_n$ is given by $\Delta_{\phi_n}^{it}=e^{-ita_n}(\id_{\ell^2(G)}\otimes \Delta_\phi^{it})(\cdot)e^{ita_n}$. Since $(T_t)$ commutes with $(\Delta_\phi^{it})$ and $e^{a_n}\in L(G)\otimes \IC 1_{\HH}$, the semigroup $(\id_{\ell^2(G)}\otimes T_t)$ commutes with $(\Delta_{\phi_n}^{it})$.

Since $M_n$ is the centralizer of $\phi_n$, the subspace $H_n$ is the fixed-point set of $(\Delta_{\tilde\phi_n}^{it})$. In particular, $(\id_{\ell^2(G)}\otimes T_t)(H_n)\subset H_n$. From Ouhabaz's theorem \cite[Corollary 2.4]{Ouh96} we deduce $\E\circ P_n\leq \E$.

Now Lemma \ref{lem:Dirichlet_approx} shows that $\E$ is a modular completely Dirichlet form with respect to $\AA_{\tilde\phi}$.

If $\AA$ is unital with unit $1_{\AA}$, then the left and right action of $\AA$ on $\H$ are unital since they are non-degenerate by definition. Thus
\begin{equation*}
\delta(1_\AA)=1_\AA\cdot\delta(1_\AA)+\delta(1_\AA)\cdot 1_\AA-\delta(1_\AA)=0
\end{equation*}
and hence $\E(1_\AA)=0$. Thus $T_t(1_\AA)=0$, which implies that $\E$ is a quantum Dirichlet form.
\end{proof}

In the situation of the previous theorem, we call $\E$ the \emph{completely Dirichlet form associated with $\delta$}.

\begin{remark}
If $\AA$ is not unital, the completely Dirichlet form associated with a derivation is not necessarily a quantum Dirichlet form, even in the commutative case. For example, this is the case for the standard Dirichlet energy $\E(f)=\int_\Omega \abs{\nabla f}^2$ with domain $H^1_0(\Omega)\cap L^\infty(\Omega)$ if $\Omega$ is a bounded Lipschitz domain.
\end{remark}

\begin{proposition}\label{prop:rel_deriv_diff_calc}
If $\AA$ is a unital Tomita algebra, $\H$ a normal Tomita bimodule over $\AA$ and $\delta\colon\AA\to\H$ a closable symmetric derivation with associated completely Dirichlet form $\E$, then the first-order differential calculus associated with $\E$ is a corestriction of $(\H^{\mathrm{a}},\bar\delta|_{\AA_\E})$. In particular,
\begin{equation*}
\bar\delta(ab)=a\bar\delta(b)+\bar\delta(a)b
\end{equation*}
for $a,b\in\AA_\E$.
\end{proposition}
\begin{proof}
Since $\AA$ is unital, $\E$ is a modular quantum Dirichlet form by \ref{thm:main}. Let $(\H_\E,\delta_\E)$ be a first-order differential calculus associated with $\E$. By definition, $\AA\subset\AA_\E\subset\AA^{\prime\prime}$ and the graph norm of $\bar\delta$ coincides with the graph norm of $\delta_\E$ on $\AA_\E$. Thus $\pi_l(\AA)^{\prime\prime}$ is strong$^\ast$ dense in $\pi_l(\AA_\E)^{\prime\prime}$ and $\AA$ is a core for $\delta_\E$. It follows that the linear hull of $\{\delta_\E(a)b\mid a,b\in\AA\}$ is dense in $\H_\E$.

Let
\begin{equation*}
U\colon\mathrm{lin}\{\delta_\E(a)b\mid a,b\in\AA\}\to\H,\,U(\delta_\E(a)b)=\delta(a)b.
\end{equation*}
By \cite[Lemma 5.9]{Wir22b} the map $U$ is well-defined and extends to an isometric $\AA_\E$-bimodule map from $\bar\H_\E$ to $\bar\H$ such that $U(\delta_\E(a))=\delta(a)$ for $a\in\AA$. 

If $a\in\AA_\E$, let $(a_n)$ be a sequence in $\AA$ such that $\norm{a_n-a}_{\bar\delta}\to 0$. As discussed above, this implies $\delta_\E(a_n)\to \delta_\E(a)$. Hence $U(\delta_\E(a))=\bar\delta(a)$. If $a,b\in\AA_\E$, then
\begin{align*}
\bar\delta(ab)&=U(\delta_\E(ab))\\
&=U(a\delta_\E(b)+\delta_\E(a)b)\\
&=a U(\delta_\E(b))+U(\delta_\E(a))b\\
&=a\bar\delta(b)+\bar\delta(a)b.
\end{align*}
Moreover, $\delta J=\J\delta$ can be extended by continuity to $\bar\delta J=\J\bar\delta$, and
\begin{equation*}
\IC\to \bar\H,\,z\mapsto \bar\delta(U_z a)
\end{equation*}
is an entire continuation of $t\mapsto \U_t \bar\delta(a)$ for $a\in\AA_\E$ by \cite[Lemma 6.6]{Wir22b}.

Thus $\bar\delta(\AA_\E)\subset\H^{\mathrm{a}}$ and $\bar\delta$ is a symmetric derivation on $\AA_\E$. The statement now follows from the uniqueness of the first-order differential calculus associated with a modular completely Dirichlet form \cite[Theorem 6.9]{Wir22b}.
\end{proof}

\begin{remark}
The previous result holds more generally with the same proof if $\AA$ is not necessarily unital, but the completely Dirichlet form associated with $\delta$ is still a quantum Dirichlet form.
\end{remark}

\begin{remark}
In the light of Lemma \ref{lem:diff_calc_delta_complete}, one has $\widehat{\AA}^\delta\subset \AA_\E$ in the situation of the previous proposition. It is an interesting question if one always has equality or if different derivations with $\delta$-complete domains can have the same associated completely Dirichlet form.
\end{remark}

\section{Examples}\label{sec:examples}

In this section we present several classes of derivations that give rise to modular completely Dirichlet forms according to Theorem \ref{thm:main}. The first three classes of examples concern inner derivations, before we treat derivations arising in non-tracial free probability in Example \ref{ex:deriv_free_prob} and derivations induced by cocycles on locally compact groups in Example~\ref{ex:deriv_cocycle}.

\begin{example}
Let $\AA$ be a Tomita algebra, $\H$ a normal Tomita bimodule over $\AA$ and $\xi\in \H$ be a bounded vector. Assume that there exists $\omega\in \IR$ such that $\U_t \xi=e^{i\omega t}\xi$ for all $t\in \IR$.
The map 
\begin{equation*}
\delta\colon \AA\to\H\oplus \H,\,a\mapsto i(\xi a-a\xi,(\J\xi)a-a(\J\xi))
\end{equation*}
is a bounded derivation, and it is symmetric when $\H\oplus\H$ is endowed with the involution $(\eta,\zeta)\mapsto (\J\zeta,\J\eta)$ and the complex one-parameter group $(e^{-i\omega z}\U_z,e^{i\omega z}\U_z)$.

It follows that the closure of the quadratic form
\begin{equation*}
\AA\to [0,\infty),\,a\mapsto \norm{\xi a-a\xi}_\H^2+\norm{(\J\xi)a-a(\J\xi)}_\H^2
\end{equation*}
is a (bounded) modular completely Dirichlet form with respect to $\AA^{\prime\prime}$. In the case $\H=\AA$, this was first proven by Cipriani \cite[Proposition 5.3]{Cip97}. See also \cite[Proposition 5.5]{Wir22b} for arbitrary Tomita bimodules $\H$ over $\AA$.
\end{example}

The next example is a (partial) extension of the previous example allowing for vectors implementing the inner derivation that are not necessarily bounded.

\begin{example}
Let $\AA$ be a Tomita algebra with Hilbert completion $\HH$. For $\xi\in \dom(\Delta^{1/2})$ the operator
\begin{equation*}
\pi_l^0(\xi)\colon \AA\to \HH,\,a\mapsto \xi a
\end{equation*}
is closable since $\pi_l^0(\xi^\sharp)\subset \pi_l^0(\xi)^\ast$. Likewise, if $\xi\in\dom(\Delta^{-1/2})$, then
\begin{equation*}
\pi_r^0(\xi)\colon \AA\to \HH,\,a\mapsto a\xi
\end{equation*}
is closable with $\pi_r^0(\xi^\flat)\subset \pi_r^0(\xi)^\ast$. Hence if $\xi\in \dom(\Delta^{-1/2})\cap \dom(\Delta^{1/2})$, then $\pi_l^0(\xi)-\pi_r^0(\xi)$ is closable with $\pi_l^0(\xi^\sharp)-\pi_r^0(\xi^\flat)\subset (\pi_l^0(\xi)-\pi_r^0(\xi))^\ast$.

Now assume that there exists $\omega\in\IR$ such that $\Delta^{it}\xi=e^{i\omega t}\xi$ for all $t\in\IR$. This implies in particular $\xi\in\dom(\Delta^{-1/2})\cap\dom(\Delta^{1/2})$.

Similar to the last example, one can turn $\AA\oplus\AA$ into a Tomita bimodule over $\AA$ if one equips it with the usual bimodule structure, the involution $(\eta,\zeta)\mapsto (J\zeta,J\eta)$ and the complex one-parameter group $(e^{-i\omega z}U_z,e^{i\omega z}U_z)$.Then the map
\begin{equation*}
\delta\colon \AA\to\AA\oplus \AA,\,a\mapsto i(\xi a-a\xi,(J\xi)a-a(J\xi))
\end{equation*}
is a closable symmetric derivation.

Thus the closure of the quadratic form
\begin{equation*}
\AA\to [0,\infty),\,a\mapsto \norm{\xi a-a\xi}_\HH^2+\norm{(J\xi)a-a(J\xi)}_\HH^2
\end{equation*}
is a modular completely Dirichlet form with respect to $\AA^{\prime\prime}$. This result has first been obtained by Cipriani and Zegarlinski \cite[Theorem 2.5]{CZ21}.
\end{example}

The previous examples require eigenvectors of the modular group to construct a symmetric derivation, which may be hard to find. In the following examples we show that in certain situations one can start with an arbitrary element if one ``averages'' the action of the modular group to ensure modularity.
\begin{example}
Let $M$ be a von Neumann algebra with separable predual. A normal semi-finite weight faithful weight $\phi$ on $M$ is called \emph{integrable} \cite[Definition II.2.1]{CT77} if
\begin{equation*}
\q_\phi=\left\{x\in M: \int_\IR \sigma^\phi_t(x^\ast x)\,dt\text{ exists in the $\sigma$-strong topology}\right\}
\end{equation*}
is weak$^\ast$ dense in $M$.

If $\phi$ is integrable, the set
\begin{equation*}
\AA=\{\Lambda_\phi(x)\mid x\in M\text{ analytic for }\sigma^\phi, \sigma^\phi_z(x)\in \q_\phi\cap \q_\phi^\ast\cap \n_\phi\cap \n_\phi^\ast\text{ for all }z\in\IC\}
\end{equation*}
is a Tomita subalgebra of $(\AA_\phi)_0$ with Hilbert completion $L^2(M)$ and $\pi_l(\AA)^{\prime\prime}=M$, as can be seen from \cite[Lemma II.2.3]{CT77} together with a standard mollifying argument.

Let $(V_t)$ be the translation group on $L^2(\IR)$, that is, $V_t f(s)=f(s+t)$, and let $L^2(\IR)^{\mathrm{a}}$ be the set of all entire analytic elements for $(V_t)$. Endow $L^2(\IR)^{\mathrm{a}}\odot \AA$ with the left and right action of $\AA$ given by $a(f\otimes b)c=f\otimes abc$, the complex one-parameter group $(V_z\odot U_z)_{z\in\IC}$ and the involution $f\otimes a\mapsto \bar f\otimes Ja$. It can be checked that this makes $L^2(\IR)^{\mathrm{a}}\odot \AA$ into a normal Tomita bimodule, which we denote by $\H$.

Let $a\in \dom(\Delta_\phi^{1/2})\cap\dom(\Delta_\phi^{-1/2})$ with $Ja=a$ and define
\begin{equation*}
\delta\colon \AA\to L^2(\IR;L^2(M)),\,\delta(b)(s)=(U_{-s}a)b-b(U_{-s}a).
\end{equation*}
We have
\begin{align*}
\delta(U_t b)(s)&=(U_{-s}a)(U_t b)-(U_t b)(U_{-s}a)\\
&=U_t((U_{-(s+t)}a)b-b(U_{-(s+t)}a))\\
&=U_t \delta(b)(s+t).
\end{align*}
Thus $\delta\circ U_t=(U_t\otimes V_t)\circ\delta$. In particular, $\delta$ maps into $\H^{\mathrm{a}}$.

It is not hard to check that $\delta\colon \AA\to \H^{\mathrm{a}}$ is a symmetric derivation. To show closability, first note that for every fixed $s\in \IR$ the map $b\mapsto\delta(b)(s)$ is closable as seen in the previous example. If $b_n\to 0$ and $\delta(b_n)\to \xi$, then there exists a subsequence such that $\delta(b_{n_k})(s)\to \xi(s)$ for a.e. $s\in \IR$. Closability of the map $b\mapsto \delta(b)(s)$ implies $\xi(s)=0$ for a.e. $s\in \IR$, which proves the closability of $\delta$.

Thus the closure of the quadratic form
\begin{equation*}
\AA\to [0,\infty),\,b\mapsto \int_\IR \norm{(U_{-s}a)b-b(U_{-s}a)}^2\,ds
\end{equation*}
is a modular completely Dirichlet form.

If we drop the assumption $Ja=a$, a similar argument shows that the closure of the quadratic form
\begin{equation*}
\AA\to [0,\infty),\,b\mapsto\int_\IR(\norm{(U_{-s}a)b-b(U_{-s}a)}^2+\norm{(U_{-s}Ja)b+b(U_{-s}Ja)}^2)\,ds
\end{equation*}
is a modular completely Dirichlet form with respect to $\AA^{\prime\prime}$.

A similar construction is possible if one starts with a weight with periodic modular group instead of an integrable weight and integrates over a period of the modular group.
\end{example}

The following class of examples of derivations was introduced by Nelson \cite{Nel17} in the context of non-tracial free probability.
\begin{example}\label{ex:deriv_free_prob}
Let $M$ be a von Neumann algebra, $\phi$ a normal faithful state on $M$ and $B\subset M$ a $\ast$-subalgebra. Let $\partial\colon B\to M\overline{\otimes} M$ be a linear map such that
\begin{equation*}
\partial(xy)=(x\otimes 1)\cdot\partial(y)+\partial(x)\cdot (1\otimes y)
\end{equation*}
for $x,y\in B$. Note that Nelson works with $M\overline{\otimes} M^\op$ instead, but under the identification $x\otimes y\mapsto x\otimes y^\op$, the $M$-bimodules $M\overline{\otimes} M$ and $M\overline{\otimes} M^\op$ (with the bimodule structure used in \cite{Nel17}) are isomorphic.

Let $\omega\in\IR$ and write $M_\infty$ for the set of entire analytic elements for $\sigma^\phi$. Nelson \cite[Definition 3.2]{Nel17} calls the map $\partial$ an $e^\omega$-modular derivation if $B\subset M_\infty$, $B$ is invariant under $\sigma^\phi_z$ for all $z\in\IC$, $\partial(B)\subset M_\infty\odot M_\infty$ and
\begin{equation*}
\partial(\sigma^\phi_z(x))=e^{i\omega z}(\sigma^\phi_z\otimes\sigma^{\phi}_z)(\partial(x))
\end{equation*}
for all $x\in B$ and $z\in \IC$.

One example given by Nelson is the free difference quotient from free probability (see \cite[Definition 3.4]{Nel17} in the non-tracial case). Given a $\ast$-subalgebra $B$ of $M$ and an element $a\in M$ that is algebraically free from $B$ (and $a^\ast$ is algebraically free from $a$ if $a\neq a^\ast)$, let
\begin{equation*}
\partial_a\colon B[a]\to B[a]\odot B[a],\,\partial_a(a)=1\otimes 1,\,\delta\vert_B=0
\end{equation*}
(and $\delta_a(a^\ast)=0$ if $a^\ast\neq a$). If $a$ is an eigenvector of $\Delta_\phi$ to the eigenvalue $e^\omega$, then $\partial_a$ is an $e^\omega$-modular derivation.

Let us see how an $e^\omega$ derivation gives rise to a symmetric derivation in our sense. For $x,y\in M$ let $(x\otimes y)^\dagger=y^\ast\otimes x^\ast$. The conjugate derivation of $\partial$ is the map 
\begin{equation*}
\hat\partial\colon B\to M_\infty\odot M_\infty,\,\hat\partial(x)=\partial(x^\ast)^\dagger.
\end{equation*}

Let $\AA=\Lambda_\phi(B)$. Since $B$ is consists of the analytic elements for $\sigma^\phi$ and is invariant under $\sigma^\phi_z$ for $z\in \IC$, the set $\AA$ is a Tomita subalgebra of $(\AA_\phi)_0=\Lambda_\phi(M_\infty)$.

Let $\H=(\Lambda_\phi(M_\infty)\odot\Lambda_\phi(M_\infty))^{\oplus 2}$ with left and right action of $\AA$ given by
\begin{equation*}
a(\xi_1\otimes\eta_1,\xi_2\otimes \eta_2)b=(a\xi_1 \otimes\eta_1 b,a\xi_2 \otimes \eta_2 b),
\end{equation*}
 involution $\J$ given by $(\xi_1\otimes\eta_1,\xi_2\otimes\eta_2)\mapsto (J\eta_2\otimes J\xi_2,J\eta_1\otimes J\xi_1)$, and complex one-parameter group $(\U_z)=(e^{i\omega z}\Delta_{\phi\otimes\phi}^{iz},e^{-i\omega z}\Delta_{\phi\otimes\phi}^{iz})$. One can check that this makes $\H$ into a normal Tomita bimodule over $\AA$.
 
Let 
\begin{equation*}
\delta\colon \AA\to\H,\,\delta(\Lambda_\phi(x))=(\Lambda_{\phi\otimes\phi}(\partial(x)),\Lambda_{\phi\otimes \phi}(\partial^\dagger(x))).
\end{equation*}
The product rule for $\partial$ and $\hat\partial$ translate to the product rule for $\delta$, the $e^\omega$ modularity of $\partial$ ensures $\delta\circ \Delta_\phi^{iz}=\U_z \circ\delta$ and the definition of $\hat\partial$ and $\J$ are tailored to guarantee $\delta\circ J_\phi=\J\circ\delta$.

All of these properties follow by routine calculations, let us just show the product rule (for the first component of) $\delta$ as illustration. Let $\delta_1(\Lambda_\phi(x))=\Lambda_{\phi\otimes\phi}(\partial(x))$. By the product rule for $\partial$ we have
\begin{align*}
\delta_1(\Lambda_\phi(xy))&=\Lambda_{\phi\otimes\phi}((x\otimes 1)\partial(y)+\partial(x)(1\otimes y))\\
&=\Lambda_{\phi\otimes \phi}((x\otimes 1)\Lambda_{\phi\otimes \phi}(\partial(y))+\partial(x))(1\otimes \sigma^\phi_{-i/2}(y))\\
&=(\pi_l(\Lambda_\phi(x))\otimes 1)\delta_1(\Lambda_\phi(y))+(1\otimes \pi_r(\Lambda_\phi(y)))\delta_1(\Lambda_\phi(x)).
\end{align*}

Thus, if $\delta$ is closable, the closure of the associated quadratic form is a completely Dirichlet form with respect to $\AA_\phi$ on the GNS Hilbert space $L^2(M,\phi)$.

To compare that to the result of Nelson, he showed \cite[Proposition 4.4]{Nel17} that one gets a completely Dirichlet form on the GNS Hilbert space $L^2(M_\phi,\phi)$ of the \emph{centralizer} $M_\phi$ of $\phi$, which is of course a tracial von Neumann algebra.

Our methods allow to extend this result to the ``fully'' non-tracial setting in that we obtain a modular completely Dirichlet form on the GNS Hilbert space of $M$ on not just of the centralizer. Note however that Nelson's definition of the map $\delta$ between $L^2$ spaces seems slightly different, owing to the use of $M\overline{\otimes}M^\op$ instead of $M\overline{\otimes}M$.
\end{example}

The last example concerns group von Neumann algebras. The case of discrete groups was treated in \cite[Section 10.2]{CS03}, but to cover general locally compact groups, possibly non-unimodular, one needs the theory for non-tracial reference weights as developed here.

\begin{example}\label{ex:deriv_cocycle}
Let $G$ be a locally compact group with left Haar measure $\mu$ and modular function $\Delta_G$. As discussed in \cite[Proposition VII.3.1]{Tak03}, the space $C_c(G)$ of compactly supported continuous function on $G$ with the $L^2$ inner product, the convolution product, the involution $f^\sharp(g)=\Delta_G(g)^{-1}\overline{f(g^{-1})}$ and the complex one-parameter group $U_z f(g)=\Delta_G(g)^{iz}f(g)$ forms a Tomita algebra. We write $\lambda$ and $\rho$ for the associated left and right action of $C_c(G)$ on $L^2(G)$ and $\AA_G$ for the associated full left Hilbert algebra.

Let $\pi$ be a strongly continuous orthogonal representation of $G$ on the real Hilbert space $H$. A continuous map $b\colon G\to H$ is called $1$-cocycle if $b(gh)=b(g)+\pi(g)b(h)$ for all $g,h\in G$. We extend $\pi$ to a unitary representation of $G$ on the complexification $H^\IC$ of $H$ and write $\xi\mapsto\bar \xi$ for the anti-unitary involution induced by $H\subset H^\IC$.

On $C_c(G;H^\IC)$ define a left and right action of $C_c(G)$ by
\begin{align*}
(f\ast \xi)(g)&=\int_G f(h)\pi(h)\xi(h^{-1}g)\,d\mu(h)\\
(\xi\ast f)(g)&=\int_G f(h^{-1}g)\xi(h)\,d\mu(h),
\end{align*}
an anti-unitary involution by $(\J\xi)(g)=-\Delta_G(g)^{-1/2}\pi(g)\overline{\xi(g^{-1})}$ and a complex one-parameter group by $\U_z \xi(g)=\Delta_G(g)^{iz}\xi(g)$. One can check that $C_c(G;H^\IC)$ with this operations is a Tomita bimodule over $C_c(G)$.

Let
\begin{equation*}
\delta\colon C_c(G)\to C_c(G;H^\IC),\,\delta(f)(g)=f(g)b(g).
\end{equation*}
Using the cocycle property of $b$, one gets
\begin{align*}
\delta(f_1\ast f_2)(g)&=\int_G f_1(h)f_2(h^{-1}g)\,d\mu(h)\,b(g)\\
&=\int_G f_1(h)f_2(h^{-1}g)(\pi(h)b(h^{-1}g)+b(h))\,d\mu(h)\\
&=(f_1\ast \delta(f_2))(g)+(\delta(f_1)\ast f_2)(g).
\end{align*}
It is readily verified that $\delta$ also satisfies $\delta\circ J=\J\circ\delta$ and $\delta\circ U_z=\U_z\circ \delta$ for all $z\in \IC$. Hence $\delta$ is a symmetric derivation. As a multiplication operator, it is clearly closable.

Therefore,
\begin{equation*}
\E\colon L^2(G,\mu)\to [0,\infty],\,\E(f)=\int_G \abs{f(g)}^2\norm{b(g)}^2\,d\mu(g)
\end{equation*}
is a modular completely Dirichlet form with respect to $\AA_G$. The associated quantum dynamical semigroup on $L(G)$ is given by
\begin{equation*}
P_t\left(\int_G \hat x(g)\lambda(g)\,d\mu(g)\right)=\int_G e^{-t\norm{b(g)}^2}\hat x(g)\lambda(g)\,d\mu(g).
\end{equation*}
In this case, complete positivity of $P_t$ also follows directly from Schönberg's theorem as $g\mapsto \norm{b(g)}^2$ is a conditionally negative definite function on $G$.

\end{example}

\DeclareFieldFormat[article]{citetitle}{#1}
\DeclareFieldFormat[article]{title}{#1}

\printbibliography
\end{document}